\documentclass{article}
 
\usepackage{amsmath,amssymb}
\usepackage{amsthm}
\usepackage{enumitem}
\usepackage{multicol}
\usepackage{url} 

\usepackage{color}

\usepackage{tikz}
\tikzset{
every node/.style={circle, draw, inner sep=2pt},
every picture/.style={thick}
}
\usetikzlibrary{matrix,arrows,calc}

\newtheorem{theorem}{Theorem}
\newtheorem{lemma}[theorem]{Lemma}
\newtheorem{proposition}[theorem]{Proposition}
\newtheorem{corollary}[theorem]{Corollary}

\theoremstyle{definition}
\newtheorem{definition}[theorem]{Definition}
\newtheorem{observation}[theorem]{Observation}
\newtheorem{remark}[theorem]{Remark}
\newtheorem{example}[theorem]{Example}
\newtheorem{question}[theorem]{Question}
\newtheorem{conjecture}[theorem]{Conjecture}

\newenvironment{thm}{\begin{theorem}}{\end{theorem}}
\newenvironment{lem}{\begin{lemma}}{\end{lemma}}
\newenvironment{prop}{\begin{proposition}}{\end{proposition}}
\newenvironment{cor}{\begin{corollary}}{\end{corollary}}

\newenvironment{defn}{\begin{definition}\bgroup\rm }{\egroup\end{definition}}

\newenvironment{rem}{\begin{remark}\bgroup\rm }{\egroup\end{remark}}
\newenvironment{ex}{\begin{example}\bgroup\rm }{\egroup\end{example}}

\def \mr {\operatorname{mr}}
\def \rank {\operatorname{rank}}

\def \S {\mathcal{S}}

\def \dunion {\dot{\cup}}

\newcommand{\diag}{\operatorname{diag}}

\newcommand{\lrangle}[1]{\left\langle#1\right\rangle}

\newcommand{\mrcr}{\mr^{\rm cr}}
\newcommand{\mz}{\operatorname{mz}}

\title{Critical ideals, minimum rank and zero forcing number}
\author{Carlos~A.~Alfaro \thanks{Banco de M\'exico, Mexico City, Mexico (alfaromontufar@gmail.com, carlos.alfaro@banxico.org.mx).}
\and Jephian~C.-H.~Lin \thanks{Department of Mathematics and Statistics, University of Victoria, Victoria, BC V8W 2Y2, Canada (chinhunglin@uvic.ca)}}

\begin{document}

\maketitle

\begin{abstract}
There are profound relations between the zero forcing number and the minimum rank of a graph.
We study the relation of both parameters with a third one, the algebraic co-rank; that is defined as the largest $i$ such that the $i$-th critical ideal is trivial.  This gives a new perspective for bounding and computing these three graph parameters.
\end{abstract}



\section{Introduction}
Throughout the paper, we focus on simple graphs, except for Remark~\ref{digraphrem} and Theorem~\ref{digraph1}, which are results for digraphs.  Given a graph $G$ and a set of indeterminates $X_G=\{x_u \, : \, u\in V(G)\}$,
the {\it generalized Laplacian matrix} $L(G,X_G)$ of $G$ is the matrix whose $uv$-entry is given by
\[
L(G,X_G)_{uv}=\begin{cases}
x_u& \text{ if } u=v,\\
-m_{uv}& \text{ otherwise},
\end{cases}
\]
where $m_{uv}$ is the number of the edges between vertices $u$ and $v$.
Moreover, if $\mathcal{R}[X_G]$ is the polynomial ring over a commutative ring $\mathcal{R}$ with unity in the variables $X_G$, then the {\it critical ideals} of $G$ are the determinantal ideals given by
\[
I_i(G,X_G)=\langle {\rm minors}_i(L(G,X_G))\rangle\subseteq \mathcal{R}[X_G] \text{ for all } 1\leq i\leq n,
\]
where $n$ is the number of vertices of $G$ and ${\rm minors}_i(L(G,X_G))$ is the set of the determinants of the $i\times i$ submatrices of $L(G,X_G)$.

An ideal is said to be {\it trivial} if it is equal to $\langle1\rangle$ ($=\mathcal{R}[X]$).
The {\it algebraic co-rank} $\gamma_\mathcal{R}(G)$ of $G$ is the maximum integer $i$ for which $I_i(G,X_G)$ is trivial.
For simplicity, in the following, $\gamma(G)$ denote $\gamma_{\mathbb{R}}(G)$, where $\mathbb{R}$ is the field of real numbers.  
Note that $I_n(G,X_G)=\langle \det L(G,X_G)\rangle$ is always non-trivial, and if $d_G$ denote the degree vector, then $I_n(G,d_G)=\langle 0\rangle$.

Originally, critical ideals were defined as a generalization of the critical group, {\it a.k.a.}\/ sandpile group, see \cite{alfaval, alfacorrval, corrval}.
In \cite{alfaval2,merino} can be found an account of the main results on sandpile group.
However, it is also a generalization of several other algebraic objects like Smith group or characteristic polynomials of the adjacency and Laplacian matrices, see \cite[Section 4]{alfavalvaz} and \cite[Section 3.3]{corrval}.
Here, we explore the relations with the zero forcing number and the minimum rank.
For this, we recall these well-known concepts.

The \emph{zero forcing game} is a color-change game where vertices can be blue or white.
At the beginning, the player can pick a set of vertices $B$ and color them blue while others remain white.
The goal is to color all vertices blue through repeated applications of the \emph{color change rule}: If $x$ is a blue vertex and $y$ is the only white neighbor of $x$, then $y$ turns blue, denoted as $x\rightarrow y$.
An initial set of blue vertices $B$ is called a \emph{zero forcing set} if starting with $B$ one can make all vertices blue.
The \emph{zero forcing number} $Z(G)$ is the minimum cardinality of a zero forcing set.
The \emph{chronological list} of a zero forcing game records the forces $x_i\rightarrow y_i$ in the order of performance.

For a graph $G$ on $n$ vertices, the family $\S_\mathcal{R}(G)$ collects all $n\times n$ symmetric matrices with entries in the ring $\mathcal{R}$, whose $i,j$-entry ($i\neq j$) is nonzero whenever $i$ is adjacent to $j$ and zero otherwise.
Note that the diagonal entries can be any element in the ring $\mathcal{R}$.  
The \emph{minimum rank} $\mr_\mathcal{R}(G)$ of $G$ is the smallest possible rank among matrices in $\S_\mathcal{R}(G)$.
Here we follow \cite[Definition 1]{rankring} and define the rank of a matrix over a commutative ring with unity as the largest $k$ such that there is a nonzero $k\times k$ minor that is not a zero divisor.  In the case of $\mathcal{R}=\mathbb{Z}$, the rank over $\mathbb{Z}$ is the same as the rank over $\mathbb{R}$.
For simplicity, sometimes we will denote $\mr(G)=\mr_\mathbb{R}(G)$ and $\S(G)=\S_\mathbb{R}(G)$.  

The paper is organized as follows.
Let $\mz(G)=|V(G)|-Z(G)$.
It is known \cite{AIMZmr,cancun} that $\mz(G)\leq \mr_{\mathcal{F}}(G)$ for every graph $G$ and any field $\mathcal{F}$, where $G$ can be a simple graph or a digraph.  In Section~\ref{section:mzgamma}, we extend this relation by proving that $\mz(G)\leq \gamma_\mathcal{R}(G)$.
In general, the algebraic co-rank and the minimum rank are not comparable.
However, in Section~\ref{section:mrgamma}, we explore the relation between the minimum rank and the algebraic co-rank under several rings.
As byproduct of the Weak Nullstellensatz, we conclude that when $\mathcal{R}$ is an algebraically closed field, $\mr_\mathcal{R}(G)\leq\gamma_\mathcal{R}(G)$.
We also conjecture that $\mr_\mathbb{R}(G)\leq\gamma_\mathbb{R}(G)$.
It is also known that $\mz(T)=\mr(T)$ for any tree $T$.
In Section~\ref{section:mzmrgamma}, we complement this equation by proving that, for trees, $\mz(T)$ and $\mr(T)$ are also equal to $\gamma(G)$; similar equalities are provided for several families of graphs.  
Finally, in Section~\ref{section:graphclasses} we discuss the property that the algebraic co-rank, minimum rank and $\mz$ are monotone on induced subgraphs, and extend some classifications.

\section{Zero forcing number and algebraic co-rank}\label{section:mzgamma}
In \cite{AIMZmr} it was proved that $\mz(G)$ is bounded from above by $\mr_{\mathcal{F}}(G)$ for every graph $G$ and any field $\mathcal{F}$.
We extend this result by proving that $\mz(G)$ is also bounded by the algebraic co-rank.
\begin{thm}
\label{gammaZthm}
For every graph $G$, $\mz(G)\leq \gamma_\mathcal{R}(G)$ for any commutative ring $\mathcal{R}$ with unity.
\end{thm}
\begin{proof}
Suppose $|V(G)|=n$ and $\mz(G)=k$.
Let $B$ be a zero forcing set of $G$ of cardinality $n-k$.
Let $(a_i\rightarrow b_i)_{i=1}^k$ be a chronological list.
Set $\alpha=\{a_i\}_{i=1}^k$ and $\beta=\{b_i\}_{i=1}^k$.
Let $L(G,X_G)$ be the generalized Laplacian matrix of $G$.
Let $A$ be the submatrix of $L(G,X_G)$ induced on rows $\alpha$ and columns $\beta$.
Obtain $A'$ from $A$ such that the order of rows corresponds with $a_1,\ldots, a_k$ and the order of columns corresponds with $b_1,\ldots,b_k$.
At step $t$ when $a_t\rightarrow b_t$ is about to happen, vertices $a_1,\dots,a_{t}$ are blue and vertices $b_{t},\dots,b_{k}$ are white.
Then, $a_t$ is adjacent with $b_{t}$ and is not adjacent with any of vertices $b_{t+2},\dots,b_{k}$.
Therefore, $A'$ is a lower triangular matrix with $-1$ on all diagonal entries.
Therefore, $A$ is an $k\times k$ submatrix of $L(G,X_G)$ with $\det(A)=\pm 1$.
Consequently, $\mz(G)\leq\gamma_\mathcal{R}(G)$.
\end{proof}

\begin{rem}
\label{digraphrem}
The zero forcing number $Z(D)$ and the minimum rank $\mr(D)$ of a simple digraph $D$ (which means no loops are allowed) are defined in \cite{cancun} and showed to have $|V(G)|-Z(G)\leq \mr(D)$ for all digraph $D$.  On the other hand, the critical ideals and the algebraic co-rank $\gamma_\mathcal{R}(D)$ of a digraph are defined in \cite{corrval}.  Theorem~\ref{gammaZthm} can be extended for digraphs.  That is, 
\[\mz(D)=|V(D)|-Z(D)\leq \gamma_\mathcal{R}(D)\]
for any commutative ring $\mathcal{R}$ with unity.
\end{rem}

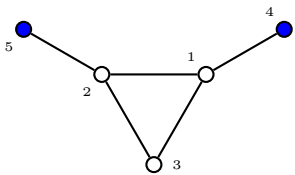
\begin{figure}[h]
\begin{center}
\begin{tikzpicture}
\foreach \i/\ang in {1/30,2/150,3/270}{
\pgfmathsetmacro{\lang}{\ang+90}
\node[label={\lang:\tiny\i}] (\i) at (\ang:0.8) {};
}
\foreach \i/\ang in {4/30,5/150}{
\pgfmathsetmacro{\lang}{\ang+90}
\node[fill=blue,label={\lang:\tiny\i}] (\i) at (\ang:2) {};
}
\draw (1) -- (2) -- (3) -- (1);
\draw (1) -- (4);
\draw (2) -- (5);
\end{tikzpicture}
\end{center}
\caption{A graph with its zero forcing set}
\label{bull}
\end{figure}

\begin{ex}
Let $G$ be the graph shown in Figure~\ref{bull}, the vertices marked as blue form a zero forcing set, and 
\[4\rightarrow 1,\qquad 5\rightarrow 2, \qquad 2\rightarrow 3\]
form a chronological list.  We have 
\[L(G,X_G)=\begin{bmatrix}
x_1 & -1 & -1 & -1 & 0 \\
-1 & x_2 & -1 & 0 & -1 \\
-1 & -1 & x_3 & 0 & 0 \\
-1 & 0 & 0 & x_4 & 0 \\
0 & -1 & 0 & 0 & x_5 \\
\end{bmatrix}.\]
If we write down the submatrix of $L(G,X_G)$ in the order of rows $(4,5,2)$ and columns $(1,2,3)$, then we get 
\[A'=\begin{bmatrix}
-1 & 0 & 0 \\
0 & -1 & 0 \\
-1 & x_2 & -1 \\
\end{bmatrix},\]
which is a lower triangular matrix with $-1$ on each diagonal entries.  Therefore, $I_1(G,X_G)=I_2(G,X_G)=I_3(G,X_G)=\lrangle{1}$ and $\gamma_{\mathcal{R}}(G)\geq 3=\mz(G)$.
\end{ex}

The main idea behind proof of Theorem \ref{gammaZthm} is to associate each zero forcing set of cardinality $k$ with a $k\times k$ submatrix of the generalized Laplacian matrix with determinant $\pm 1$.
As pointed out in Example 5.6 of \cite{alfaval}, there are graphs with algebraic co-rank $k$ having no $k$-minor equal to $\pm 1$.
Therefore, there are graph in which $\mz(G)<\gamma_\mathcal{R}(G)$.

\section{Minimum rank and algebraic co-rank}\label{section:mrgamma}

Let $I\subseteq \mathcal{R}[X]$ be an ideal in $\mathcal{R}[X]$.
The \emph{variety} of $I$ is defined as
\[
V_\mathcal{R}(I)=\left\{ {\bf a}\in \mathcal{R}^n : f({\bf a}) = 0 \text{ for all } f\in I \right\}.
\]
That is, $V_\mathcal{R}(I)$ is the set of common roots between polynomials in $I$.
It is known \cite[Proposition 4]{clo} that if $f_1, \ldots , f_s$ and $g_1, \ldots , g_t$ are two different bases of the same ideal $I$, then $V_\mathcal{R}(f_1, \ldots , f_s)=V_\mathcal{R}(g_1, \ldots , g_t)=V_\mathcal{R}(I)$.
Also, if $I$ is trivial, then $V_\mathcal{R}(I)=\emptyset$.
In terms of the critical ideals, if $I_k(G,X_G)\subseteq \mathcal{R}[X_G]$ is trivial, then, for all ${\bf a}\in \mathcal{R}^n$, there are $k$-minors of $L(G,{\bf a})$ which are different of 0, and $\rank(L(G,{\bf a}))\geq k$.
However, $\gamma_\mathcal{R}(G)=k$ does not imply that $\mr_\mathcal{R}(G)\geq k$, since matrices in $\S_\mathcal{R}(G)$ do not necessarily have only $0$ and $-1$ on the off-diagonal entries.
One property of the critical ideals \cite[Proposition 3.3]{corrval} is that
\[
\langle 1\rangle \supseteq I_1(G,X_G) \supseteq \cdots \supseteq I_n(G,X_G) \supseteq \langle 0\rangle.
\]
Thus
\[
\emptyset=V_\mathcal{R}(\langle 1\rangle) \subseteq V_\mathcal{R}(I_1(G,X_G)) \subseteq \cdots \subseteq V_\mathcal{R}(I_n(G,X_G)) \subseteq V_\mathcal{R}(\langle 0\rangle)=\mathcal{R}^n.
\]
If $V_\mathcal{R}(I_k(G,X_G))\neq\emptyset$ for some $k$, then there exists ${\bf a}\in\mathcal{R}$ such that, for all $t \geq k$, $I_{t}(G,{\bf a})=\langle 0\rangle$; that is, all $t$-minors of $L(G,{\bf a})$ are equal to $0$.
Therefore, $\mr_\mathcal{R}(G)\leq k-1$.
\begin{lem}\label{lemma:zerosetmr}
If $V_\mathcal{R}\left(I_{r+1}(G,X_G)\right)\neq\emptyset$, then $\mr_\mathcal{R}(G)\leq r$.  In particular, if $r = \gamma_\mathcal{R}(G)$ and $V_\mathcal{R}\left(I_{r+1}(G,X_G)\right)\neq\emptyset$, then $\mr_\mathcal{R}(G)\leq \gamma_\mathcal{R}(G)$.
\end{lem}

\begin{example}\label{example:cimr}
Figure~\ref{figureexamplemrcr} shows the graph $\overline{3K_2}$, which has $\gamma_\mathbb{Z}\left(\overline{3K_2}\right)=2$ and $\gamma_\mathbb{R}\left(\overline{3K_2}\right)=3$.
The Gr\"obner bases of its first non-trivial critical ideals in these rings are:
\[
I_3\left(\overline{3K_2},X_{\overline{3K_2}}\right)=\langle x_0, x_1, x_2, x_3, x_4, x_5, 2 \rangle\subseteq \mathbb{Z}\left[X_{\overline{3K_2}}\right]
\]
\begin{eqnarray*}
I_4\left(\overline{3K_2},X_{\overline{3K_2}}\right) & = & \langle x_0x_1, x_0x_2, x_0x_3 + 2x_0 + 2x_3, x_0x_4, x_0x_5, x_1x_2, \\
& & x_1x_3, x_1x_4, x_1x_5 + 2x_1 + 2x_5, x_2x_3, x_2x_4 + 2x_2 + 2x_4,\\
& & x_2x_5, x_3x_4, x_3x_5, x_4x_5 \rangle\subseteq \mathbb{R}\left[X_{\overline{3K_2}}\right]
\end{eqnarray*}
For this graph, we have that $V_\mathbb{Z}\left(I_3\left(\overline{3K_2},X_{\overline{3K_2}}\right)\right)$ is empty, that is, there is no ${\bf a}\in \mathbb{Z}^6$ such that $I_3\left(\overline{3K_2},{\bf a}\right)=\langle 0\rangle$.
On the other hand, $I_4\left(\overline{3K_2},{\bf 0}\right)=\langle 0\rangle\subseteq \mathbb{R}\left[X_{\overline{3K_2}}\right]$.
Meanwhile, $Z\left(\overline{3K_2}\right)=4$, so $\mr_\mathbb{R}\left(\overline{3K_2}\right)=\mz\left(\overline{3K_2}\right)=2$.

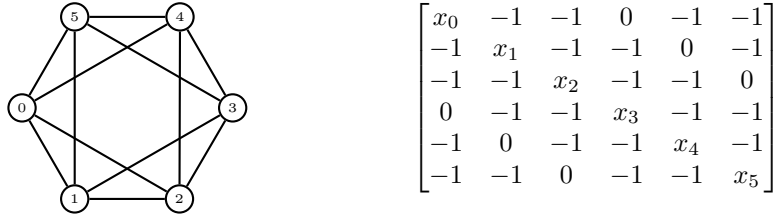
\begin{figure}[h!]
    \begin{center}
    \begin{multicols}{2}
		\begin{tikzpicture}[scale=0.7,thick]
		\tikzstyle{every node}=[minimum width=0pt, inner sep=2pt, circle]
			\draw (-2.0,0.0) node[draw] (0) { \tiny 0};
			\draw (-1.0,-1.73) node[draw] (2) { \tiny 1};
			\draw (0.99,-1.73) node[draw] (1) { \tiny 2};
			\draw (2.0,0.0) node[draw] (3) { \tiny 3};
			\draw (1.0,1.73) node[draw] (4) { \tiny 4};
			\draw (-1.0,1.73) node[draw] (5) { \tiny 5};
			\draw  (0) edge (1);
			\draw  (0) edge (2);
			\draw  (0) edge (4);
			\draw  (0) edge (5);
			\draw  (1) edge (2);
			\draw  (1) edge (3);
			\draw  (1) edge (4);
			\draw  (2) edge (3);
			\draw  (2) edge (5);
			\draw  (3) edge (4);
			\draw  (3) edge (5);
			\draw  (4) edge (5);
		\end{tikzpicture}

        $
        \begin{bmatrix}
            x_0 & -1 & -1 & 0 & -1 & -1 \\
            -1 & x_1 & -1 & -1 & 0 & -1 \\
            -1 & -1 & x_2 & -1 & -1 & 0 \\
            0 & -1 & -1 & x_3 & -1 & -1 \\
            -1 & 0 & -1 & -1 & x_4 & -1 \\
            -1 & -1 & 0 & -1 & -1 & x_5 \\
        \end{bmatrix}
        $
    \end{multicols}
	\end{center}
    \caption{The graph $\overline{3K_2}$ and its generalized Laplacian matrix}
    \label{figureexamplemrcr}
\end{figure}

\end{example}

In the case of algebraically closed fields, that is, fields where every non-constant polynomial in $\mathcal{R}[X]$ has a root in $\mathcal{R}$, minimum rank can be bounded from above by the algebraic co-rank.
This is a consequence of the following outstanding result.

\begin{lem}{\rm \cite[The Weak Nullstellensatz]{clo}}
    Let $\mathcal{R}$ be an algebraically closed field and let $I\subseteq \mathcal{R}[X]$ be an ideal satisfying $V(I)=\emptyset$.
    Then $I$ is trivial.
\end{lem}

In general, the same result for any arbitrary ring is not always true.
For instance, consider the ideal $I_3\left(\overline{3K_2},X_{\overline{3K_2}}\right)\subseteq \mathbb{Z}\left[X_{\overline{3K_2}}\right]$ in Example \ref{example:cimr}.
In this case, $V_\mathbb{Z}\left(I_3\left(\overline{3K_2},X_{\overline{3K_2}}\right)\right)$ is empty and $I_3\left(\overline{3K_2},X_{\overline{3K_2}}\right)$ is non-trivial.
However, $V_\mathbb{R}(I_4\left(\overline{3K_2},X_{\overline{3K_2}}\right))$ is not empty, in fact, $\mr_\mathbb{R}(\overline{3K_2})\leq\gamma_\mathbb{R}(\overline{3K_2})$.

Let $r = \gamma_\mathcal{R}(G)$.
If $\mathcal{R}$ is an algebraically closed field, then by the weak nullstellensatz, $V(I_{r+1}(G,X_G))$ is not empty.
And by Lemma~\ref{lemma:zerosetmr}, next result follows.

\begin{thm}\label{thm:closedringmrleqgamma}
Let $\mathcal{R}$ be an algebraically closed field.  Then $\mr_\mathcal{R}(G)\leq\gamma_\mathcal{R}(G)$ for every graph $G$.
\end{thm}

The minimum rank problem have many variants, one is that the diagonal entries might have some restrictions.
For example, a variant of the $\mr_{\mathbb{C}}$ considers only Hermitian matrices, so all the diagonal entries must be real.
However, if no restrictions exists on the diagonal entries, then Theorem~\ref{thm:closedringmrleqgamma} implies $\mr_{\mathbb{C}}(G)\leq\gamma_\mathbb{C}(G)$.  In general, it is not clear the relation between $\gamma_\mathcal{R}(G)$ and $\mr_\mathcal{R}(G)$ for any arbitrary ring $\mathcal{R}$ and graph $G$.  For $\mathcal{R}=\mathbb{R}$, we have the following conjecture.


\begin{conjecture}\label{conjecture:mrleqgamma}
For any graph $G$, $\mr_\mathbb{R}(G)\leq\gamma_\mathbb{R}(G)$.
\end{conjecture}

It is known \cite{small} that $\mr(G)=\mz(G)$ for graphs with at most $7$ vertices.
By Theorem~\ref{gammaZthm}, we know that Conjecture~\ref{conjecture:mrleqgamma} is true up to $7$ vertices.
In fact, for graphs on at most $6$ vertices, it is an equality except for $21$ graphs, see Appendix. 
For the case of integers, we do not have, in general, $\mr_\mathbb{Z}(G)\leq\gamma_\mathbb{Z}(G)$.
For instance, consider complete tripartite graphs $K_{m,n,o}$, it is known \cite[Theorem~4.2]{alfaval} that $\gamma_{\mathbb{Z}}(K_{m,n,o})\leq 2$. 
On the other hand, in \cite[Theorem~4.4]{BF07}, it was proved that when $m,n,o\geq 3$, $3=\mr_\mathbb{R}(K_{m,n,o})\leq\mr_\mathbb{Z}(K_{m,n,o})$.

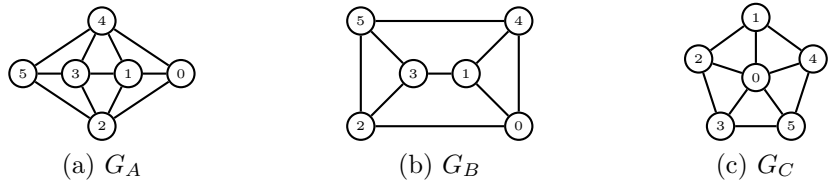
\begin{figure}[h]
\begin{center}
\begin{tabular}{c@{\extracolsep{2cm}}c@{\extracolsep{2cm}}c}
\begin{tikzpicture}[scale=.7,thick]
	\tikzstyle{every node}=[minimum width=0pt, inner sep=2pt, circle]
	\draw (1,1) node[draw] (0) {\tiny 0};
	\draw (0,1) node[draw] (1) {\tiny 1};
	\draw (-0.5,0) node[draw] (2) {\tiny 2};
	\draw (-1,1) node[draw] (3) {\tiny 3};
	\draw (-0.5,2) node[draw] (4) {\tiny 4};
	\draw (-2,1) node[draw] (5) {\tiny 5};
	\draw  (0) edge (1);
	\draw  (0) edge (2);
	\draw  (0) edge (4);
	\draw  (1) edge (2);
	\draw  (1) edge (3);
	\draw  (1) edge (4);
	\draw  (2) edge (3);
	\draw  (2) edge (5);
	\draw  (3) edge (4);
	\draw  (3) edge (5);
	\draw  (4) edge (5);
\end{tikzpicture}
&
\begin{tikzpicture}[scale=.7,thick]
	\tikzstyle{every node}=[minimum width=0pt, inner sep=2pt, circle]
	\draw (1,0) node[draw] (0) {\tiny 0};
	\draw (0,1) node[draw] (1) {\tiny 1};
	\draw (-2,0) node[draw] (2) {\tiny 2};
	\draw (-1,1) node[draw] (3) {\tiny 3};
	\draw (1,2) node[draw] (4) {\tiny 4};
	\draw (-2,2) node[draw] (5) {\tiny 5};
	\draw  (0) edge (1);
	\draw  (0) edge (2);
	\draw  (0) edge (4);
	\draw  (1) edge (3);
	\draw  (1) edge (4);
	\draw  (2) edge (3);
	\draw  (2) edge (5);
	\draw  (3) edge (5);
	\draw  (4) edge (5);
\end{tikzpicture}
&
\begin{tikzpicture}[scale=0.8,thick]
		\tikzstyle{every node}=[minimum width=0pt, inner sep=2pt, circle]
			\draw (0,0) node[draw] (0) {\tiny 0};
			\draw (90:1) node[draw] (1) {\tiny 1};
			\draw (162:1) node[draw] (2) {\tiny 2};
			\draw (234:1) node[draw] (3) {\tiny 3};
			\draw (306:1) node[draw] (4) {\tiny 5};
			\draw (378:1) node[draw] (5) {\tiny 4};
			\draw  (0) edge (1);
			\draw  (0) edge (2);
			\draw  (0) edge (3);
			\draw  (0) edge (4);
			\draw  (0) edge (5);
			\draw  (1) edge (5);
			\draw  (1) edge (2);
			\draw  (2) edge (3);
			\draw  (3) edge (4);
			\draw  (4) edge (5);
		\end{tikzpicture}
\\
(a) $G_A$ & (b) $G_B$ & (c) $G_C$ \\
\end{tabular}
\end{center}
\label{fig:3graphs}
\caption{The only three graphs with at most 6 vertices for which there exists no ${\bf a}\in \mathbb{Z}^n$ such that the first non-trivial critical ideal vanishes at ${\bf a}$}
\end{figure}

A natural approach to Conjecture~\ref{conjecture:mrleqgamma} is to use Lemma~\ref{lemma:zerosetmr}.  We have verified by \textit{Sage} \cite{sage} that for any graph with at most $6$ vertices, there exists an ${\bf a}\in \{-2,-1,0,1,2\}^n$ such that the first non-trivial critical ideal over $\mathbb{R}$ vanishes at ${\bf a}$, except for only 3 graphs; see Figure~\ref{fig:3graphs}.
For these 3 graphs ($G_A$, $G_B$ and $G_C$) there is an ${\bf a}\in \mathbb{R}^6$ such that the first non-trivial ideal vanishes at ${\bf a}$.  This can be done by looking at the Gr\"obner bases of their first non-trivial critical ideals.
We may compute $\gamma_\mathbb{R}(G_A)=\gamma_\mathbb{R}(G_B)=\gamma_\mathbb{R}(G_C)=3$, and the Gr\"obner bases of their first non-trivial critical ideals are as follows:
\begin{eqnarray*} 
I_4(G_A,X_{G_A}) & = & \langle x_0x_1 - x_1 - 2, x_0x_3 + 2x_0 + x_3, x_0x_5 + 1, \\
& & x_1x_3 + x_1 + x_3 + 2, x_1x_5 + x_1 + 2x_5, x_2,\\
& & x_3x_5 - x_3 - 2, x_4\rangle,
\end{eqnarray*}
\[I_4(G_B,X_{G_B})=\langle x_0 + x_5 - 1, x_1 + x_5 - 1, x_2 - x_5, x_3 - x_5, x_4 + x_5 - 1, x_5^2 - x_5 - 1\rangle,\]
and
\[I_4(G_C,X_{G_C})=\langle x_0 + x_5 + 3, x_1 - x_5, x_2 - x_5, x_3 - x_5, x_4 - x_5, x_5^2 + x_5 - 1\rangle.\]

\begin{thm}
    For every graph $G$ with at most 6 vertices and $r=\gamma_\mathbb{R}(G)$, there exists an ${\bf a}\in \mathbb{R}^n$ such that $\langle 0\rangle=I_{r+1}(G,{\bf a})\subseteq \mathbb{R}[X_G]$.
\end{thm}


A new variant of the minimum rank problem is to restrict the matrices to the evaluations of the generalized Laplacian matrix.
In the following we will give few examples to this problem.

\begin{defn}
Let $G$ be a graph.  The \emph{critical minimum rank} $\mrcr_\mathcal{R}(G)$ is defined as the minimum rank over $L(G,{\bf d})$ for all ${\bf d}\in \mathcal{R}^n$.  In the case of $\mathcal{R}=\mathbb{R}$, we simply write $\mrcr(G)$ for $\mrcr_\mathbb{R}(G)$.
\end{defn}

By definition, the critical minimum rank is a restricted version of the minimum rank, so $\mr_\mathcal{R}(G)\leq \mrcr_\mathcal{R}(G)$. 
 It is also true that $\gamma_\mathcal{R}(G)\leq \mrcr_\mathcal{R}(G)$ for and any graph $G$.  Figure~\ref{paramdiag} helps us visualize the relations between $\mz(G)$, $\mr_\mathcal{R}(G)$, $\gamma_\mathcal{R}(G)$, and $\mrcr_\mathcal{R}(G)$.
 
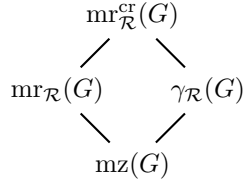
\begin{figure}[h]
\begin{center}
\begin{tikzpicture}[every node/.style={draw=none,rectangle}]
\node (mz) at (0,0) {$\mz(G)$};
\node (mr) at (-1,1) {$\mr_\mathcal{R}(G)$};
\node (gamma) at (1,1) {$\gamma_\mathcal{R}(G)$};
\node (mrcr) at (0,2) {$\mrcr_\mathcal{R}(G)$};
\draw (mz) -- (mr);
\draw (mz) -- (gamma);
\draw (mrcr) -- (mr);
\draw (mrcr) -- (gamma);
\end{tikzpicture}
\end{center}
\caption{The relations between $\mz(G)$, $\mr_\mathcal{R}(G)$, $\gamma_\mathcal{R}(G)$, and $\mrcr_\mathcal{R}(G)$, where each line means the upper parameter is an upper bound for the lower parameter}
\label{paramdiag}
\end{figure}

The inequality $\mr(G)\leq \mrcr(G)$ can be strict.  For example, we know $\mr(\overline{3K_2})=2$ and $\mrcr(\overline{3K_2})\geq \gamma(\overline{3K_2})=3$ by Example~\ref{example:cimr}.

\begin{prop}
\label{crprop}
Let $G$ be a graph.  If there is a vector ${\bf d}\in \mathcal{R}^n$ such that $\rank(L(G,{\bf d}))=\mr_\mathcal{R}(G)$, then $\mr_\mathcal{R}(G)=\mrcr_\mathcal{R}(G)$.  Specifically, if there is a vector ${\bf d}\in\mathbb{Z}^n$ such that $\rank(L(G,{\bf d}))=\mz(G)$, then $\mz(G)=\mrcr_\mathbb{Z}(G)$ and 
\[\mz(G)=\mr(G)=\mr_\mathbb{Z}(G)=\gamma(G)=\gamma_\mathbb{Z}(G)=\mrcr_\mathbb{Z}(G).\]
\end{prop}
\begin{proof}
If there is a vector ${\bf d}\in \mathcal{R}^n$ such that $\rank(L(G,{\bf d}))=\mr_\mathcal{R}(G)$, then 
\[\mrcr_\mathcal{R}(G)\leq \rank(L(G,{\bf d}))=\mr_\mathcal{R}(G).\]
Through the diagram in Figure~\ref{paramdiag}, we know $\mr_\mathcal{R}(G)= \mrcr_\mathcal{R}(G)$.

Similarly, if there is a vector ${\bf d}\in\mathbb{Z}^n$ such that $\rank(L(G,{\bf d}))=\mz(G)$, then 
\[\mrcr_\mathbb{Z}(G)\leq \rank(L(G,{\bf d}))=\mz(G).\]
Through the diagram in Figure~\ref{paramdiag}, we know $\mr_\mathbb{Z}(G)= \mz(G)$.  Since 
\[\mz(G)\leq \mr_\mathbb{R}(G)\leq \mr_\mathbb{Z}(G)\leq \mrcr_\mathbb{Z}(G)\]
and 
\[\mz(G)\leq \gamma_\mathbb{Z}(G)\leq \gamma_\mathbb{R}(G)\leq \mrcr_\mathbb{R}(G)\leq \mrcr_\mathbb{Z}(G),\]
we know all the mentioned quantities are the same.
\end{proof}

The next section we will use Proposition~\ref{crprop} to show that $\mz(G)=\gamma(G)=\mr(G)$ for many graphs.

\section{Graphs with $\mz(G)=\gamma(G)=\mr(G)$}\label{section:mzmrgamma}

We begin by showing $\mz(G)=\gamma(G)=\mr(G)$ when $G$ is a tree and combining results in the literature.

The minimum rank of a tree has been well-studied.  In \cite{LDJ99}, it was shown that $M(T)=P(T)=\Delta(T)$ for any tree $T$.  Here $M(T)$ is the \emph{maximum nullity} of $T$ over $\mathbb{R}$, which is equal to $|V(T)|-\mr(T)$; the \emph{path cover number} $P(T)$ is the minimum number of disjoint induced paths on $T$  that can cover the vertices of $T$; the parameter $\Delta(T)$ is defined as the maximum of $p-q$ such that by deleting $q$ vertices from $T$ the remaining graph becomes $p$ paths.  In \cite{AIMZmr} it was proved that $\mr(T)=\mz(T)$ for any tree $T$.  It is also known \cite{CDHMP07} that the minimum rank of a tree is field independent; that is, $\mr(T)=\mr_\mathcal{F}(T)$ for any field $\mathcal{F}$.  In summary, for any tree $T$ on $n$ vertices, it is known that 
\[\mr(T)=\mz(T)=n-P(T)=n-\Delta(T)=\mr_\mathcal{F}(T)\]
for any field $\mathcal{F}$.  

Next, we will show that $\mz(T)=\mrcr_\mathbb{Z}(T)$ for any tree $T$.
\begin{thm}
\label{treeequal}
For any tree $T$, $\mz(T)=\mrcr_\mathbb{Z}(T)$, which is the same as $\mr(T)$, $\mr_\mathbb{Z}(T)$, $\gamma(T)$, $\gamma_\mathbb{Z}(T)$, $n-P(T)$, and $n-\Delta(T)$.
\end{thm}

\begin{proof}
It is known \cite{AIMZmr} that $\mr(T)=\mz(T)$.  Let $A$ be a matrix in $\S(T)$ such that $\rank(A)=\mr(T)=\mz(T)$.  By \cite[Theorem 4.2]{Hogben05}, $A$ can be chosen as a $0,1$-matrix.  Let ${\bf d}_A$ be the vector that records the diagonal of $A$.  Then $-A=L(T,-{\bf d}_A)$.  By Proposition~\ref{crprop}, $\mz(T)=\mrcr_\mathbb{Z}(T)$, so the related quantities are the same.
\end{proof}

For a graph $G$, a \emph{$2$-matching} is a set edges $\mathcal{M}\subseteq E(G)$ such that every vertex of $G$ is incident to at most two edges in $\mathcal{M}$.  One may think of $\mathcal{M}$ as the edges of a disjoint union of (not necessarily induced) paths as a subgraph of $G$.  The \emph{$2$-matching number} $\nu_2(G)$ is the maximum cardinality (number of edges) of a $2$-matching of $G$.  In \cite{corrval1} it was proved that $\gamma_\mathbb{Z}(T)=\nu_2(T)$ for any tree $T$.  By Theorem~\ref{treeequal}, we know $\nu_2(T)=\gamma_\mathbb{Z}(T)=n-P(T)$ for any tree $T$ on $n$ vertices.  Below we give a direct proof of this result.

\begin{prop}
For any tree $T$ on $n$ vertices, $\nu_2(T)=n-P(T)$.
\end{prop}
\begin{proof}
Let $\mathcal{M}$ be a maximum $2$-matching of $T$.  Let $X$ be the set of vertices that are incident to at least one edge in $\mathcal{M}$.  Since $\mathcal{M}$ is maximum, $G-X$ should be some isolated vertices.  Consider $\mathcal{P}$ as a collection of path that includes the paths in $\mathcal{M}$ and the isolated vertices (each as a path on one vertex) in $G-X$.  Since $T$ is a tree, every path in $\mathcal{P}$ is an induced path.  The paths in $\mathcal{P}$ covers all vertices of $T$ and has $|\mathcal{M}|$ edges.  Thus, there are $n-|\mathcal{M}|$ paths in $\mathcal{P}$ and $P(T)\leq n-|\mathcal{M}|=n-\nu_2(T)$.

Conversely, each path cover of $T$ with $P(T)$ paths has $n-P(T)$ edges, and these edges form a $2$-matching of $T$, so $\nu_2(T)\geq n-P(T)$.
\end{proof}

\begin{cor}
There is a linear-time algorithm for finding $\gamma_\mathbb{Z}(T)$, which is the same as $\nu_2(T)$, $\mz(T)$, $\mr(T)$, $\mr_\mathbb{Z}(T)$, $\gamma(T)$, $n-P(T)$, and $n-\Delta(T)$.
\end{cor}
\begin{proof}
In \cite{JS02}, there is a linear-time algorithm for finding $\Delta(T)$.  By Theorem~\ref{treeequal}, all the mentioned quantities can be found in linear-time.
\end{proof}

\begin{prop}
For any cycle $C_n$ with $n\geq 3$, $\mz(C_n)=\mrcr_\mathbb{Z}(C_n)=n-2$, which is the same as $\mr(C_n)$, $\mr_\mathbb{Z}(C_n)$, $\gamma(C_n)$, and $\gamma_\mathbb{Z}(C_n)$.
\end{prop}
\begin{proof}
By \cite[Theorem 4.5]{Hogben05}, for each $n\neq 5$, there is a $0,1$-matrix $A$ such that $\rank(A)=\mz(C_n)=n-2$.  Let ${\bf d}_A$ be the vector recording the diagonal entries of $A$, then $\rank(L(C_n,-{\bf d}_A))=n-2$.  For $n=5$, we may pick ${\bf d}=(0,-1,1,1,2)$, where the vertices are labeled with respect to the cycle order.  Thus, $\rank(L(C_5,{\bf d}))=3=n-2$.  Consequently, we may apply Proposition~\ref{crprop} to get the desired results.
\end{proof}

\begin{prop}
The Petersen graph $G$ has $\mz(G)=\mrcr_\mathbb{Z}(G)=5$, which is the same as $\mr(G)$, $\mr_\mathbb{Z}(G)$, $\gamma(G)$, and $\gamma_\mathbb{Z}(G)$.
\end{prop}
\begin{proof}
By \cite[Proposition 2.8]{DGHMR09}, the adjacency matrix $A$ of the Petersen graph has $\rank(A-I)=5=\mr(G)=\mz(G)$.  Therefore, the desired results follow from Proposition~\ref{crprop}.
\end{proof}

\begin{prop}
Let $G$ be the line graph of a tree.  Then $\mz(G)=\mrcr_\mathbb{Z}(G)$, which is the same as $\mr(G)$, $\mr_\mathbb{Z}(G)$, $\gamma(G)$, and $\gamma_\mathbb{Z}(G)$.
\end{prop}
\begin{proof}
By \cite[Corollary 2.10 and Corollary 2.11]{DGHMR09}, there is a vector ${\bf d}$ such that $\rank(A+\diag({\bf d}))=\mr(G)=\mz(G)$.  By applying Proposition~\ref{crprop} to $L(G,-{\bf d})$, we get the desired results.
\end{proof}

\section{Graph classes for bounded $\mz$, $\mr$ and $\gamma$}\label{section:graphclasses}

It is known that algebraic co-rank, minimum rank and $\mz$ are monotone on induced subgraphs.
Then, it is natural to ask for classifications of graphs where these parameters are bounded from above.
In this section we explore some relations between previous characterization, and extend one for directed graphs.
First, we recall the following result.
\begin{lem}
\label{monotone}
{\rm \cite{AIMZmr,corrval}}
Let $H$ be an induced subgraph of $G$. Then,
\begin{enumerate}[label={\rm (\arabic*)}]
    \item $\gamma_\mathcal{R}(H)\leq \gamma_\mathcal{R}(G)$,
    \item $\mr_\mathcal{R}(H)\leq \mr_\mathcal{R}(G)$,
    \item $\mz(H)\leq \mz(G)$.
\end{enumerate}
\end{lem}
Given a family of graphs $\mathfrak{F}$, a graph $G$ is called $\mathfrak{F}$-{\it free} if no induced subgraph of $G$ is isomorphic to a member of $\mathfrak{F}$.
Since $\mz(G)\leq \gamma_{\mathcal{R}}(G)$ and $\mz(G)\leq \mr_{\mathcal{R}}(G)$, then the family of graphs with $\gamma_{\mathcal{R}}(G)\leq k$ or $\mr_{\mathcal{R}}(G)\leq k$ are contained in the family of graphs with $\mz(G)\leq k$.
In previous works, it was noticed that among connected graphs $K_n$ is the unique graph with minimum rank, algebraic co-rank and $\mz$ equal to 1.

\begin{thm}{\rm \cite{alfaval,BHL04}}
\label{lemma:mzandmrandgamma=1}
    Let $G$ be a connected graph and $\mathcal{R}$ a commutative ring with unity.
    Then, the following are equivalent:
    \begin{enumerate}[label={\rm (\arabic*)}]
        \item $G$ is the complete graph,
        \item $G$ is $P_3$-free,
        \item $\mr_\mathcal{R}(G)\leq 1$,
        \item $\gamma_\mathcal{R}(G)\leq 1$,
        \item $\mz(G)\leq 1$.
    \end{enumerate}
\end{thm}

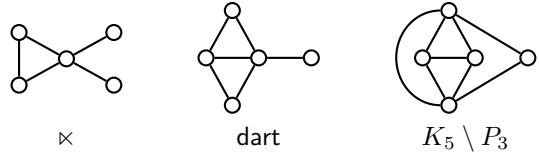
\begin{figure}[h]
\begin{center}
\begin{tabular}{c@{\extracolsep{10mm}}c@{\extracolsep{10mm}}c@{\extracolsep{10mm}}c@{\extracolsep{10mm}}c}
	\begin{tikzpicture}[rotate=-90,scale=.7]
	\tikzstyle{every node}=[minimum width=0pt, inner sep=2pt, circle]
	\draw (-.5,-.9) node (v1) [draw] {};
	\draw (.5,-.9) node (v2) [draw] {};
	\draw (0,0) node (v3) [draw] {};
	\draw (-.5,.9) node (v4) [draw] {};
	\draw (.5,.9) node (v5) [draw] {};
	\draw (v1) -- (v2);
	\draw (v1) -- (v3);
	\draw (v2) -- (v3);
	\draw (v3) -- (v4);
	\draw (v3) -- (v5);
    \node[draw=none] at (0.9,0) {};
	\end{tikzpicture}
&
	\begin{tikzpicture}[scale=.7]
	\tikzstyle{every node}=[minimum width=0pt, inner sep=2pt, circle]
	\draw (-.5,0) node (v2) [draw] {};
	\draw (0,-.9) node (v1) [draw] {};
	\draw (.5,0) node (v3) [draw] {};
	\draw (1.5,0) node (v5) [draw] {};
	\draw (0,.9) node (v4) [draw] {};
	\draw (v1) -- (v2);
	\draw (v1) -- (v3);
	\draw (v2) -- (v3);
	\draw (v2) -- (v4);
	\draw (v3) -- (v4);
	\draw (v3) -- (v5);
	\end{tikzpicture}
    &
	\begin{tikzpicture}[scale=.7]
	\tikzstyle{every node}=[minimum width=0pt, inner sep=2pt, circle]
	\draw (-.5,0) node (v2) [draw] {};
	\draw (0,-.9) node (v1) [draw] {};
	\draw (.5,0) node (v3) [draw] {};
	\draw (1.5,0) node (v5) [draw] {};
	\draw (0,.9) node (v4) [draw] {};
	\draw (v1) -- (v2);
	\draw (v1) -- (v3);
	\draw (v2) -- (v3);
	\draw (v2) -- (v4);
	\draw (v3) -- (v4);
	\draw (v4) -- (v5);
    \draw (v1) -- (v5);
    \draw (v4) to [out=180,in=90] ($(v2)+(-0.5,0)$) to [out=-90, in=180 ] (v1);
	\end{tikzpicture}
\\
$\ltimes$
&
{\sf dart}
&
$K_5\setminus{P_3}$
\end{tabular}
\end{center}
\caption{The graphs $\ltimes$, ${\sf dart}$ and $K_5\setminus{P_3}$}
\label{fig2}
\end{figure}


This confirms Conjecture \ref{conjecture:mrleqgamma} for graphs with $\mr_{\mathbb{R}}(G)\leq 2$.

\begin{theorem}
    If $G$ is a connected graph such that $\mr_{\mathbb{R}}(G)\leq 2$, then $\mr_{\mathbb{R}}(G)\leq\gamma_{\mathbb{R}}(G)$.
\end{theorem}
\begin{proof}
    Let $G$ be a connected graph with $\mr_{\mathbb{R}}(G)\leq 2$.
    If $G$ has $\gamma_{\mathbb{R}}(G)\leq 1$, then $G$ is a complete graph.
    Then, by Theorem~\ref{lemma:mzandmrandgamma=1}, $\mr_{\mathbb{R}}(G)=\gamma_{\mathbb{R}}(G)=1$.
    On the other hand, if $G$ has $\mr_{\mathbb{R}}(G)= 2$, then $\gamma_{\mathbb{R}}(G)\geq2$.
\end{proof}

Finally, we turn our attention to simple digraphs, where loops and multiedges are not allowed.  The concept of the algebraic co-rank , the minimum rank, and the zero forcing number extend to digraphs naturally \cite{corrval,cancun} with slightly modifications.  Let $D$ be a digraph on $n$ vertices. 
The generalized Laplacian matrix $L(D,X_D)$ is defined in the same way as that of a simple graph, except that $m_{uv}$ is the number of arcs going from $u$ to $v$.  
For the minimum rank of a digraph, the family $\S_\mathcal{R}(G)$ consists of all $n\times n$ matrices with entries in the ring $\mathcal{R}$ whose $i,j$-entry ($i\neq j$) is nonzero whenever $(i,j)$ is an arc and zero otherwise.
For zero forcing number of a digraph, the color change rule is applied in $y$ when $y$ is the only {\it out-neighbor} of the blue vertex $x$.  By defining these concepts, Lemma~\ref{monotone} can be extended to digraphs.  That is, if $D_1$ is an induced subdigraph of $D_2$, then $\gamma_\mathcal{R}(D_1)\leq \gamma_\mathcal{R}(D_2)$, $\mr_\mathcal{R}(D_1)\leq\mr_\mathcal{R}(D_2)$, and $\mz(D_1)\leq \mz (D_2)$.

In the rest of this section, we contribute by extending a classification of digraphs with at most 1 trivial critical ideal, obtained in \cite{alfavalvaz}, by including the minimum rank and the parameter $\mz(D)$.


\begin{figure}[h!]
\begin{center}
\begin{tikzpicture}[scale=.7,>=stealth',shorten >=1pt,bend angle = 15]
 	\tikzstyle{every node}=[minimum width=0pt, inner sep=2pt, circle]
 	\draw (30:1.5) node[draw, fill] (1) {\bf\color{white} $K_{n_2}$};
 	\draw (150:1.5) node[draw] (2) {$T_{n_1}$};
 	\draw (270:1.5) node[draw] (3) {$T'_{n_3}$};
 	\draw[->] (1) -- (3);
 	\draw[->] (2) -- (1);
 	\draw[->] (2) -- (3);
 \end{tikzpicture}
\end{center}
\caption{The digraph $\Lambda_{n_1,n_2,n_3}$}
\label{figure:lambda}
\end{figure}
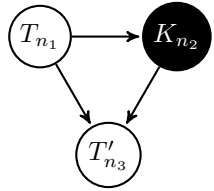
 
Let $\Lambda_{n_1,n_2,n_3}$ be the digraph defined in the following way:
The vertex set $V(\Lambda_{n_1,n_2,n_3})$ is partitioned in three sets $T$, $T'$ and $K$ with $n_1,n_3$ and $n_2$ vertices, respectively, 
such that $T$ and $T'$ are two trivial digraphs (which have no arcs), and $K$ is a complete digraph (which has double arcs between each pair of vertices).
Additionally, the arc sets $(T,K)_{\Lambda_{n_1,n_2,n_3}}$, $(T,T')_{\Lambda_{n_1,n_2,n_3}}$ and $({K},T')_{\Lambda_{n_1,n_2,n_3}}$ are complete.
See Figure~\ref{figure:lambda} for a graphical representation of $\Lambda_{n_1,n_2,n_3}$.

Let $\mathfrak{F}$ be the family of the seventeen digraphs.  

\begin{figure}[h!]
\[
   \begin{tikzpicture}[scale=.6,>=stealth',shorten >=1pt,bend angle = 15]
   \tikzstyle{every node}=[minimum width=0pt, inner sep=1pt, circle]
		\draw (180:1) node (v1) [draw,fill] {};
		\draw (300:1) node (v2) [draw] {};
		\draw (420:1) node (v3) [draw] {};
		\draw (-90:1.5) node (name) {\small $F_{3,1}$};
		\draw[->] (v1) edge (v3);
		\draw[->] (v3) edge (v2);
   \end{tikzpicture}
		\hspace{1cm}
   \begin{tikzpicture}[scale=.6,>=stealth',shorten >=1pt,bend angle = 15]
   \tikzstyle{every node}=[minimum width=0pt, inner sep=1pt, circle]
		\draw (180:1) node (v1) [draw] {};
		\draw (300:1) node (v2) [draw,fill] {};
		\draw (420:1) node (v3) [draw] {};
		\draw (-90:1.5) node (name) {\small $F_{3,2}$};
		\draw[->,bend right] (v1) edge (v3);
		\draw[->] (v2) edge (v3);
		\draw[->,bend right] (v3) edge (v1);
   \end{tikzpicture}
		\hspace{1cm}
   \begin{tikzpicture}[scale=.6,>=stealth',shorten >=1pt,bend angle = 15]
   \tikzstyle{every node}=[minimum width=0pt, inner sep=1pt, circle]
		\draw (180:1) node (v1) [draw,fill] {};
		\draw (300:1) node (v2) [draw] {};
		\draw (420:1) node (v3) [draw] {};
		\draw (-90:1.5) node (name) {\small $F_{3,3}$};
		\draw[->,bend right] (v1) edge (v3);
		\draw[->,bend right] (v3) edge (v1);
		\draw[->] (v3) edge (v2);
   \end{tikzpicture}
		\hspace{1cm}
   \begin{tikzpicture}[scale=.6,>=stealth',shorten >=1pt,bend angle = 15]
   \tikzstyle{every node}=[minimum width=0pt, inner sep=1pt, circle]
		\draw (180:1) node (v1) [draw,fill] {};
		\draw (300:1) node (v2) [draw] {};
		\draw (420:1) node (v3) [draw] {};
		\draw (-90:1.5) node (name) {\small $F_{3,4}$};
		\draw[->,bend right] (v1) edge (v3);
		\draw[->,bend right] (v2) edge (v3);
		\draw[->,bend right] (v3) edge (v1);
		\draw[->,bend right] (v3) edge (v2);
   \end{tikzpicture}
		\hspace{1cm}
   \begin{tikzpicture}[scale=.6,>=stealth',shorten >=1pt,bend angle = 15]
   \tikzstyle{every node}=[minimum width=0pt, inner sep=1pt, circle]
		\draw (180:1) node (v1) [draw,fill] {};
		\draw (300:1) node (v2) [draw] {};
		\draw (420:1) node (v3) [draw] {};
		\draw (-90:1.5) node (name) {\small $F_{3,5}$};
		\draw[->] (v1) edge (v2);
		\draw[->] (v2) edge (v3);
		\draw[->] (v3) edge (v1);
   \end{tikzpicture}
		\hspace{1cm}
   \begin{tikzpicture}[scale=.6,>=stealth',shorten >=1pt,bend angle = 15]
   \tikzstyle{every node}=[minimum width=0pt, inner sep=1pt, circle]
		\draw (180:1) node (v1) [draw] {};
		\draw (300:1) node (v2) [draw,fill] {};
		\draw (420:1) node (v3) [draw] {};
		\draw (-90:1.5) node (name) {\small $F_{3,6}$};
		\draw[->,bend right] (v1) edge (v2);
		\draw[->] (v1) edge (v3);
		\draw[->,bend right] (v2) edge (v1);
		\draw[->] (v3) edge (v2);
   \end{tikzpicture}
\]
\[
   \begin{tikzpicture}[scale=.6,>=stealth',shorten >=1pt,bend angle = 15]
   \tikzstyle{every node}=[minimum width=0pt, inner sep=1pt, circle]
		\draw (180:1) node (v1) [draw] {};
		\draw (300:1) node (v2) [draw] {};
		\draw (420:1) node (v3) [draw,fill] {};
		\draw (-90:1.5) node (name) {\small $F_{3,6}$};
		\draw[->,bend right] (v1) edge (v2);
		\draw[->,bend right] (v1) edge (v3);
		\draw[->,bend right] (v2) edge (v1);
		\draw[->] (v2) edge (v3);
		\draw[->,bend right] (v3) edge (v1);
   \end{tikzpicture}
		\hspace{1cm}
   \begin{tikzpicture}[scale=.6,>=stealth',shorten >=1pt,bend angle = 15]
   \tikzstyle{every node}=[minimum width=0pt, inner sep=1pt, circle]
		\draw (180:1) node (v1) [draw,fill] {};
		\draw (270:1) node (v3) [draw] {};
		\draw (360:1) node (v2) [draw,fill] {};
		\draw (450:1) node (v4) [draw] {};
		\draw (-90:1.5) node (name) {\small $F_{4,1}$};
		\draw[->] (v1) edge (v3);
		\draw[->] (v1) edge (v4);
		\draw[->] (v2) edge (v4);
   \end{tikzpicture}
		\hspace{1cm}
   \begin{tikzpicture}[scale=.6,>=stealth',shorten >=1pt,bend angle = 15]
   \tikzstyle{every node}=[minimum width=0pt, inner sep=1pt, circle]
		\draw (180:1) node (v1) [draw,fill] {};
		\draw (300:1) node (v2) [draw,fill] {};
		\draw (420:1) node (v3) [draw] {};
		\draw (0,0) node (v4) [draw] {};
		\draw (-90:1.5) node (name) {\small $F_{4,2}$};
		\draw[->] (v1) edge (v3);
		\draw[->] (v1) edge (v4);
		\draw[->] (v2) edge (v4);
		\draw[->] (v3) edge (v4);
   \end{tikzpicture}
		\hspace{1cm}
   \begin{tikzpicture}[scale=.6,>=stealth',shorten >=1pt,bend angle = 15]
   \tikzstyle{every node}=[minimum width=0pt, inner sep=1pt, circle]
		\draw (180:1) node (v1) [draw,fill] {};
		\draw (300:1) node (v2) [draw] {};
		\draw (420:1) node (v3) [draw] {};
		\draw (0,0) node (v4) [draw,fill] {};
		\draw (-90:1.5) node (name) {\small $F_{4,3}$};
		\draw[->] (v1) edge (v3);
		\draw[->] (v4) edge (v1);
		\draw[->] (v4) edge (v2);
		\draw[->] (v4) edge (v3);
   \end{tikzpicture}
		\hspace{1cm}
   \begin{tikzpicture}[scale=.6,>=stealth',shorten >=1pt,bend angle = 15]
   \tikzstyle{every node}=[minimum width=0pt, inner sep=1pt, circle]
		\draw (180:1) node (v1) [draw,fill] {};
		\draw (300:1) node (v2) [draw,fill] {};
		\draw (420:1) node (v3) [draw] {};
		\draw (300:0.2) node (v4) [draw] {};
		\draw (-90:1.5) node (name) {\small $F_{4,4}$};
		\draw[->,bend right] (v1) edge (v3);
		\draw[->] (v1) edge (v4);
		\draw[->] (v2) edge (v4);
		\draw[->,bend right] (v3) edge (v1);
		\draw[->] (v3) edge (v4);
   \end{tikzpicture}
		\hspace{1cm}
   \begin{tikzpicture}[scale=.6,>=stealth',shorten >=1pt,bend angle = 15]
   \tikzstyle{every node}=[minimum width=0pt, inner sep=1pt, circle]
		\draw (180:1) node (v1) [draw,fill] {};
		\draw (300:1) node (v2) [draw] {};
		\draw (420:1) node (v3) [draw] {};
		\draw (300:0.2) node (v4) [draw,fill] {};
		\draw (-90:1.5) node (name) {\small $F_{4,5}$};
		\draw[->,bend right] (v1) edge (v3);
		\draw[->,bend right] (v3) edge (v1);
		\draw[->] (v4) edge (v1);
		\draw[->] (v4) edge (v2);
		\draw[->] (v4) edge (v3);
   \end{tikzpicture}
\]
\[
   \begin{tikzpicture}[scale=.6,>=stealth',shorten >=1pt,bend angle = 15]
   \tikzstyle{every node}=[minimum width=0pt, inner sep=1pt, circle]
		\draw (180:1) node (v1) [draw,fill] {};
		\draw (300:1) node (v2) [draw] {};
		\draw (420:1) node (v3) [draw] {};
		\draw (0,0) node (v4) [draw,fill] {};
		\draw (-90:1.5) node (name) {\small $F_{4,6}$};
		\draw[->] (v1) edge (v3);
		\draw[->] (v2) edge (v3);
		\draw[->] (v4) edge (v1);
		\draw[->] (v4) edge (v2);
		\draw[->] (v4) edge (v3);
   \end{tikzpicture}
		\hspace{1cm}
   \begin{tikzpicture}[scale=.6,>=stealth',shorten >=1pt,bend angle = 15]
   \tikzstyle{every node}=[minimum width=0pt, inner sep=1pt, circle]
		\draw (180:1) node (v1) [draw,fill] {};
		\draw (300:1) node (v2) [draw] {};
		\draw (420:1) node (v3) [draw,fill] {};
		\draw (0,0) node (v4) [draw] {};
		\draw (-90:1.5) node (name) {\small $F_{4,7}$};
		\draw[->] (v1) edge (v2);
		\draw[->] (v1) edge (v3);
		\draw[->] (v1) edge (v4);
		\draw[->] (v2) edge (v3);
		\draw[->] (v2) edge (v4);
		\draw[->] (v3) edge (v4);
   \end{tikzpicture}
		\hspace{1cm}
   \begin{tikzpicture}[scale=.6,>=stealth',shorten >=1pt,bend angle = 15]
   \tikzstyle{every node}=[minimum width=0pt, inner sep=1pt, circle]
		\draw (180:1) node (v1) [draw,fill] {};
		\draw (300:1) node (v2) [draw] {};
		\draw (420:0.2) node (v3) [draw,fill] {};
		\draw (420:1) node (v4) [draw] {};
		\draw (-90:1.5) node (name) {\small $F_{4,8}$};
		\draw[->,bend right] (v1) edge (v2);
		\draw[->] (v1) edge (v3);
		\draw[->] (v1) edge (v4);
		\draw[->,bend right] (v2) edge (v1);
		\draw[->] (v2) edge (v3);
		\draw[->] (v2) edge (v4);
		\draw[->] (v3) edge (v4);
   \end{tikzpicture}
		\hspace{1cm}
   \begin{tikzpicture}[scale=.6,>=stealth',shorten >=1pt,bend angle = 15]
   \tikzstyle{every node}=[minimum width=0pt, inner sep=1pt, circle]
		\draw (180:1) node (v1) [draw,fill] {};
		\draw (270:1) node (v2) [draw] {};
		\draw (360:1) node (v3) [draw] {};
		\draw (450:1) node (v4) [draw,fill] {};
		\draw (-90:1.5) node (name) {\small $F_{4,9}$};
		\draw[->,bend right] (v1) edge (v2);
		\draw[->] (v1) edge (v3);
		\draw[->] (v1) edge (v4);
		\draw[->,bend right] (v2) edge (v1);
		\draw[->] (v2) edge (v3);
		\draw[->] (v2) edge (v4);
		\draw[->,bend right] (v3) edge (v4);
		\draw[->,bend right] (v4) edge (v3);
   \end{tikzpicture}
		\hspace{1cm}
   \begin{tikzpicture}[scale=.6,>=stealth',shorten >=1pt,bend angle = 15]
   \tikzstyle{every node}=[minimum width=0pt, inner sep=1pt, circle]
		\draw (180:1) node (v1) [draw,fill] {};
		\draw (300:1) node (v2) [draw] {};
		\draw (420:1) node (v3) [draw,fill] {};
		\draw (420:0.2) node (v4) [draw] {};
		\draw (-90:1.5) node (name) {\small $F_{4,10}$};
		\draw[->,bend right] (v1) edge (v2);
		\draw[->,bend right] (v2) edge (v1);
		\draw[->] (v3) edge (v1);
		\draw[->] (v3) edge (v2);
		\draw[->] (v3) edge (v4);
		\draw[->] (v4) edge (v1);
		\draw[->] (v4) edge (v2);
   \end{tikzpicture}
\]
\caption{Seventeen digraphs with algebraic co-rank equal to $2$, where the filled vertices mark a zero forcing set for each graph}
\label{fig:ForbDig1}
\end{figure}
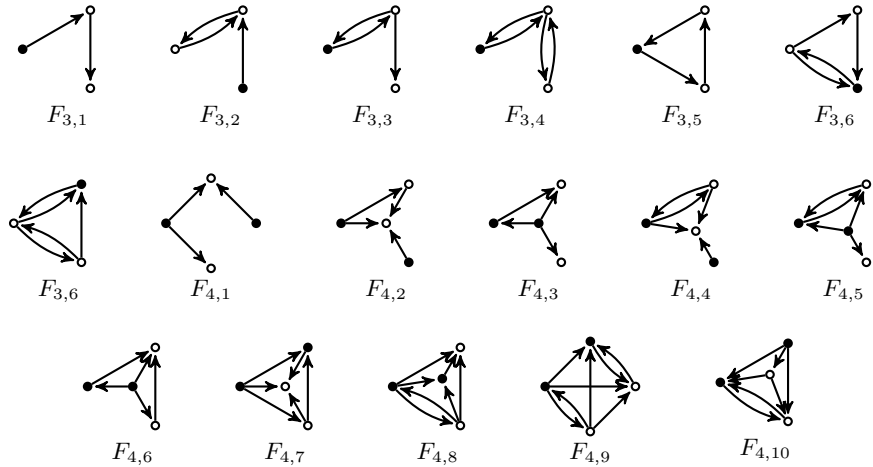 

\begin{thm}
\label{digraph1}
Let $\mathcal{R}$ be a commutative ring with unity.
The following are equivalent:
\begin{enumerate}[label={\rm (\arabic*)}]
\item $D$ is $\mathfrak{F}$-free,
\item $D$ is isomorphic to $\Lambda_{n_1,n_2,n_3}$,
\item $\mr_{\mathcal{R}}(D)\leq 1$,
\item $\mz(D)\leq 1$,
\item $\gamma_{\mathcal{R}}(D)\leq 1$.
\end{enumerate}
\end{thm}
\begin{proof}
In \cite{alfavalvaz}, it was proven that (1), (2), and $\gamma_\mathbb{Z}(D)\leq 1$ are equivalent.  

Next we will show that (2) implies (3) and (5). 
Suppose $D$ is a digraph isomorphic to $\Lambda_{n_1,n_2,n_3}$.
Then
\[\begin{bmatrix}
O & J & J \\ 
O & J & J \\
O & O & O
\end{bmatrix}\]
is a matrix in $\S_\mathcal{R}(D)$, with the partition $V(D)=V(T_{n_1})\dunion V(K_{n_2})\dunion V(T'_{n_3})$.  Since this matrix has rank $1$, then $\mr_\mathcal{R}(D)\leq 1$.
And by Proposition \ref{crprop} $\gamma_\mathcal{R}(D)\leq 1$.

Since $\mz(D)\leq \mr_{\mathcal{R}}(D)$ and $\mz(D)\leq \gamma_{\mathcal{R}}(D)$ for all digraphs, either of (3) or (5) implies (4).

Finally, Figure~\ref{fig:ForbDig1} shows a zero forcing set for each of the seventeen digraphs in $\mathfrak{F}$, so $\mz(D)=2$ for all $D\in\mathfrak{F}$.  Therefore, (4) implies (1).
\end{proof}


\section*{Acknowledgments}
C.A. Alfaro was partially supported by SNI and CONACyT.


\section*{Appendix: Graphs with at most 6 vertices and $\mz(G)\neq\gamma_\mathbb{R}(G)$}

From the 143 connected graphs with at most 6 vertices, only 21 graphs have $\mz(G)< \gamma_\mathbb{R}(G)$.
For the other graphs, $\mz(G)=\gamma_\mathbb{Z}(G)=\gamma_\mathbb{R}(G)$.

\begin{center}
\begin{tabular}{ccccccc}
&
\begin{tikzpicture}[scale=.3]
		\tikzstyle{every node}=[minimum width=0pt, inner sep=2pt, circle]
			\draw (-2.0,0.0) node[draw] (0) {};
			\draw (-0.61,-1.9) node[draw] (1) {};
			\draw (1.61,-1.17) node[draw] (2) {};
			\draw (1.61,1.17) node[draw] (3) {};
			\draw (-0.61,1.9) node[draw] (4) {};
			\draw  (0) edge (1);
			\draw  (0) edge (2);
			\draw  (0) edge (3);
			\draw  (0) edge (4);
			\draw  (1) edge (2);
			\draw  (1) edge (3);
			\draw  (2) edge (3);
			\draw  (3) edge (4);
		\end{tikzpicture}
&
\begin{tikzpicture}[scale=.3,thick]
		\tikzstyle{every node}=[minimum width=0pt, inner sep=2pt, circle]
			\draw (-2.0,0.0) node[draw] (0) {};
			\draw (-1.0,-1.73) node[draw] (4) {};
			\draw (0.99,-1.73) node[draw] (1) {};
			\draw (2.0,0.0) node[draw] (5) {};
			\draw (1.0,1.73) node[draw] (2) {};
			\draw (-1.0,1.73) node[draw] (3) {};
			\draw  (0) edge (1);
			\draw  (0) edge (2);
			\draw  (0) edge (3);
			\draw  (0) edge (4);
			\draw  (1) edge (2);
			\draw  (1) edge (4);
			\draw  (2) edge (3);
			\draw  (2) edge (4);
			\draw  (2) edge (5);
		\end{tikzpicture}
&
		\begin{tikzpicture}[scale=0.3,thick]
		\tikzstyle{every node}=[minimum width=0pt, inner sep=2pt, circle]
			\draw (-2.0,0.0) node[draw] (5) {};
			\draw (-1.0,-1.73) node[draw] (4) {};
			\draw (0.99,-1.73) node[draw] (2) {};
			\draw (2.0,0.0) node[draw] (3) {};
			\draw (1.0,1.73) node[draw] (1) {};
			\draw (-1.0,1.73) node[draw] (0) {};
			\draw  (0) edge (1);
			\draw  (0) edge (2);
			\draw  (0) edge (4);
			\draw  (1) edge (2);
			\draw  (1) edge (3);
			\draw  (1) edge (4);
			\draw  (2) edge (3);
			\draw  (2) edge (4);
			\draw  (4) edge (5);
		\end{tikzpicture}
&
		\begin{tikzpicture}[scale=.3,thick]
		\tikzstyle{every node}=[minimum width=0pt, inner sep=2pt, circle]
			\draw (-2.0,0.0) node[draw] (5) {};
			\draw (-1.0,-1.73) node[draw] (0) {};
			\draw (0.99,-1.73) node[draw] (2) {};
			\draw (2.0,0.0) node[draw] (3) {};
			\draw (1.0,1.73) node[draw] (4) {};
			\draw (-1.0,1.73) node[draw] (1) {};
			\draw  (0) edge (1);
			\draw  (0) edge (2);
			\draw  (0) edge (3);
			\draw  (0) edge (5);
			\draw  (1) edge (4);
			\draw  (2) edge (3);
			\draw  (3) edge (4);
			\draw  (4) edge (5);
		\end{tikzpicture}
&
\begin{tikzpicture}[scale=.3,thick]
		\tikzstyle{every node}=[minimum width=0pt, inner sep=2pt, circle]
			\draw (-2.0,0.0) node[draw] (5) {};
			\draw (-1.0,-1.73) node[draw] (0) {};
			\draw (0.99,-1.73) node[draw] (2) {};
			\draw (2.0,0.0) node[draw] (3) {};
			\draw (1.0,1.73) node[draw] (4) {};
			\draw (-1.0,1.73) node[draw] (1) {};
			\draw  (0) edge (1);
			\draw  (0) edge (2);
			\draw  (0) edge (3);
			\draw  (0) edge (4);
			\draw  (0) edge (5);
			\draw  (1) edge (4);
			\draw  (2) edge (3);
			\draw  (3) edge (4);
			\draw  (4) edge (5);
		\end{tikzpicture}
&

		\begin{tikzpicture}[scale=.3,thick]
		\tikzstyle{every node}=[minimum width=0pt, inner sep=2pt, circle]
			\draw (-2.0,0.0) node[draw] (1) {};
			\draw (-1.0,-1.73) node[draw] (0) {};
			\draw (0.99,-1.73) node[draw] (2) {};
			\draw (2.0,0.0) node[draw] (3) {};
			\draw (1.0,1.73) node[draw] (5) {};
			\draw (-1.0,1.73) node[draw] (4) {};
			\draw  (0) edge (1);
			\draw  (0) edge (2);
			\draw  (0) edge (3);
			\draw  (1) edge (2);
			\draw  (1) edge (3);
			\draw  (1) edge (4);
			\draw  (2) edge (3);
			\draw  (3) edge (5);
			\draw  (4) edge (5);
		\end{tikzpicture}

\\
\small $\mz$ & \small 2 & \small 3 & \small 3 & \small 3 & \small 3 & \small 3\\
\small $\gamma_\mathbb{Z}$ & \small 3 & \small 4 & \small 4 & \small 4 & \small 4 & \small 4\\
\small $\gamma_\mathbb{R}$ & \small 3 & \small 4 & \small 4 & \small 4 & \small 4 & \small 4\\
\end{tabular}
\end{center}

\begin{center}
\begin{tabular}{ccccccc}
&
		\begin{tikzpicture}[scale=.3,thick]
		\tikzstyle{every node}=[minimum width=0pt, inner sep=2pt, circle]
			\draw (-2.0,0.0) node[draw] (3) {};
			\draw (-1.0,-1.73) node[draw] (1) {};
			\draw (0.99,-1.73) node[draw] (2) {};
			\draw (2.0,0.0) node[draw] (0) {};
			\draw (1.0,1.73) node[draw] (4) {};
			\draw (-1.0,1.73) node[draw] (5) {};
			\draw  (0) edge (1);
			\draw  (0) edge (2);
			\draw  (0) edge (3);
			\draw  (0) edge (4);
			\draw  (1) edge (2);
			\draw  (1) edge (3);
			\draw  (1) edge (4);
			\draw  (2) edge (3);
			\draw  (2) edge (4);
			\draw  (3) edge (4);
			\draw  (3) edge (5);
			\draw  (4) edge (5);
		\end{tikzpicture}
&
		\begin{tikzpicture}[scale=.3,thick]
		\tikzstyle{every node}=[minimum width=0pt, inner sep=2pt, circle]
			\draw (-2.0,0.0) node[draw] (0) {};
			\draw (-1.0,-1.73) node[draw] (1) {};
			\draw (0.99,-1.73) node[draw] (2) {};
			\draw (2.0,0.0) node[draw] (3) {};
			\draw (1.0,1.73) node[draw] (5) {};
			\draw (-1.0,1.73) node[draw] (4) {};
			\draw  (0) edge (1);
			\draw  (0) edge (2);
			\draw  (0) edge (3);
			\draw  (0) edge (4);
			\draw  (1) edge (2);
			\draw  (1) edge (3);
			\draw  (1) edge (4);
			\draw  (2) edge (3);
			\draw  (3) edge (4);
			\draw  (3) edge (5);
			\draw  (4) edge (5);
		\end{tikzpicture}
&
		\begin{tikzpicture}[scale=.3,thick]
		\tikzstyle{every node}=[minimum width=0pt, inner sep=2pt, circle]
			\draw (-2.0,0.0) node[draw] (1) {};
			\draw (-1.0,-1.73) node[draw] (0) {};
			\draw (0.99,-1.73) node[draw] (2) {};
			\draw (2.0,0.0) node[draw] (3) {};
			\draw (1.0,1.73) node[draw] (5) {};
			\draw (-1.0,1.73) node[draw] (4) {};
			\draw  (0) edge (1);
			\draw  (0) edge (2);
			\draw  (0) edge (3);
			\draw  (0) edge (4);
			\draw  (1) edge (2);
			\draw  (1) edge (4);
			\draw  (2) edge (3);
			\draw  (3) edge (4);
			\draw  (3) edge (5);
			\draw  (4) edge (5);
		\end{tikzpicture}
&
		\begin{tikzpicture}[scale=.3,thick]
		\tikzstyle{every node}=[minimum width=0pt, inner sep=2pt, circle]
			\draw (-2.0,0.0) node[draw] (1) {};
			\draw (-1.0,-1.73) node[draw] (0) {};
			\draw (0.99,-1.73) node[draw] (2) {};
			\draw (2.0,0.0) node[draw] (3) {};
			\draw (1.0,1.73) node[draw] (4) {};
			\draw (-1.0,1.73) node[draw] (5) {};
			\draw  (0) edge (1);
			\draw  (0) edge (2);
			\draw  (0) edge (3);
			\draw  (1) edge (2);
			\draw  (1) edge (3);
			\draw  (1) edge (4);
			\draw  (1) edge (5);
			\draw  (2) edge (3);
			\draw  (3) edge (4);
			\draw  (3) edge (5);
			\draw  (4) edge (5);
		\end{tikzpicture}
&
		\begin{tikzpicture}[scale=.3,thick]
		\tikzstyle{every node}=[minimum width=0pt, inner sep=2pt, circle]
			\draw (-2.0,0.0) node[draw] (1) {};
			\draw (-1.0,-1.73) node[draw] (0) {};
			\draw (0.99,-1.73) node[draw] (2) {};
			\draw (2.0,0.0) node[draw] (3) {};
			\draw (1.0,1.73) node[draw] (5) {};
			\draw (-1.0,1.73) node[draw] (4) {};
			\draw  (0) edge (1);
			\draw  (0) edge (2);
			\draw  (0) edge (3);
			\draw  (0) edge (4);
			\draw  (1) edge (3);
			\draw  (1) edge (4);
			\draw  (2) edge (3);
			\draw  (3) edge (5);
			\draw  (4) edge (5);
		\end{tikzpicture}
&
		\begin{tikzpicture}[scale=.3,thick]
		\tikzstyle{every node}=[minimum width=0pt, inner sep=2pt, circle]
			\draw (-2.0,0.0) node[draw] (1) {};
			\draw (-1.0,-1.73) node[draw] (0) {};
			\draw (0.99,-1.73) node[draw] (2) {};
			\draw (2.0,0.0) node[draw] (3) {};
			\draw (1.0,1.73) node[draw] (5) {};
			\draw (-1.0,1.73) node[draw] (4) {};
			\draw  (0) edge (1);
			\draw  (0) edge (2);
			\draw  (0) edge (3);
			\draw  (0) edge (4);
			\draw  (1) edge (3);
			\draw  (1) edge (4);
			\draw  (2) edge (3);
			\draw  (3) edge (4);
			\draw  (3) edge (5);
			\draw  (4) edge (5);
		\end{tikzpicture}

\\
\small $\mz$ & \small 2 & \small 3 & \small 3 & \small 2 & \small 3 & \small 3\\
\small $\gamma_\mathbb{Z}$ & \small 3 & \small 4 & \small 4 & \small 3 & \small 4 & \small 4\\
\small $\gamma_\mathbb{R}$ & \small 3 & \small 4 & \small 4 & \small 3 & \small 4 & \small 4\\
\end{tabular}
\end{center}

\begin{center}
\begin{tabular}{ccccccc}
&
		\begin{tikzpicture}[scale=.3,thick]
		\tikzstyle{every node}=[minimum width=0pt, inner sep=2pt, circle]
			\draw (-2.0,0.0) node[draw] (0) {};
			\draw (-1.0,-1.73) node[draw] (1) {};
			\draw (0.99,-1.73) node[draw] (2) {};
			\draw (2.0,0.0) node[draw] (3) {};
			\draw (1.0,1.73) node[draw] (5) {};
			\draw (-1.0,1.73) node[draw] (4) {};
			\draw  (0) edge (1);
			\draw  (0) edge (2);
			\draw  (0) edge (4);
			\draw  (1) edge (2);
			\draw  (1) edge (3);
			\draw  (1) edge (4);
			\draw  (2) edge (3);
			\draw  (2) edge (5);
			\draw  (3) edge (5);
			\draw  (4) edge (5);
		\end{tikzpicture}
&
		\begin{tikzpicture}[scale=.3,thick]
		\tikzstyle{every node}=[minimum width=0pt, inner sep=2pt, circle]
			\draw (-2.0,0.0) node[draw] (4) {};
			\draw (-1.0,-1.73) node[draw] (0) {};
			\draw (0.99,-1.73) node[draw] (1) {};
			\draw (2.0,0.0) node[draw] (3) {};
			\draw (1.0,1.73) node[draw] (2) {};
			\draw (-1.0,1.73) node[draw] (5) {};
			\draw  (0) edge (1);
			\draw  (0) edge (2);
			\draw  (0) edge (4);
			\draw  (1) edge (2);
			\draw  (1) edge (3);
			\draw  (1) edge (4);
			\draw  (2) edge (3);
			\draw  (2) edge (4);
			\draw  (2) edge (5);
			\draw  (3) edge (5);
			\draw  (4) edge (5);
		\end{tikzpicture}
&
		\begin{tikzpicture}[scale=.3,thick]
		\tikzstyle{every node}=[minimum width=0pt, inner sep=2pt, circle]
			\draw (-2.0,0.0) node[draw] (0) {};
			\draw (-1.0,-1.73) node[draw] (1) {};
			\draw (0.99,-1.73) node[draw] (2) {};
			\draw (2.0,0.0) node[draw] (3) {};
			\draw (1.0,1.73) node[draw] (5) {};
			\draw (-1.0,1.73) node[draw] (4) {};
			\draw  (0) edge (1);
			\draw  (0) edge (2);
			\draw  (0) edge (4);
			\draw  (1) edge (2);
			\draw  (1) edge (3);
			\draw  (1) edge (4);
			\draw  (2) edge (3);
			\draw  (2) edge (4);
			\draw  (2) edge (5);
			\draw  (3) edge (4);
			\draw  (3) edge (5);
			\draw  (4) edge (5);
		\end{tikzpicture}
&
		\begin{tikzpicture}[scale=.3,thick]
		\tikzstyle{every node}=[minimum width=0pt, inner sep=2pt, circle]
			\draw (-2.0,0.0) node[draw] (0) {};
			\draw (-1.0,-1.73) node[draw] (2) {};
			\draw (0.99,-1.73) node[draw] (1) {};
			\draw (2.0,0.0) node[draw] (3) {};
			\draw (1.0,1.73) node[draw] (4) {};
			\draw (-1.0,1.73) node[draw] (5) {};
			\draw  (0) edge (1);
			\draw  (0) edge (2);
			\draw  (0) edge (3);
			\draw  (0) edge (4);
			\draw  (0) edge (5);
			\draw  (1) edge (2);
			\draw  (1) edge (3);
			\draw  (1) edge (4);
			\draw  (2) edge (3);
			\draw  (2) edge (5);
			\draw  (3) edge (4);
			\draw  (3) edge (5);
			\draw  (4) edge (5);
		\end{tikzpicture}
&
		\begin{tikzpicture}[scale=.3,thick]
		\tikzstyle{every node}=[minimum width=0pt, inner sep=2pt, circle]
			\draw (-2.0,0.0) node[draw] (0) {};
			\draw (-1.0,-1.73) node[draw] (3) {};
			\draw (0.99,-1.73) node[draw] (2) {};
			\draw (2.0,0.0) node[draw] (1) {};
			\draw (1.0,1.73) node[draw] (4) {};
			\draw (-1.0,1.73) node[draw] (5) {};
			\draw  (0) edge (1);
			\draw  (0) edge (2);
			\draw  (0) edge (3);
			\draw  (0) edge (5);
			\draw  (1) edge (2);
			\draw  (1) edge (3);
			\draw  (1) edge (4);
			\draw  (2) edge (3);
			\draw  (2) edge (4);
			\draw  (3) edge (5);
			\draw  (4) edge (5);
		\end{tikzpicture}
&
		\begin{tikzpicture}[scale=0.3,thick]
		\tikzstyle{every node}=[minimum width=0pt, inner sep=2pt, circle]
			\draw (-2.0,0.0) node[draw] (5) {};
			\draw (-1.0,-1.73) node[draw] (3) {};
			\draw (0.99,-1.73) node[draw] (1) {};
			\draw (2.0,0.0) node[draw] (2) {};
			\draw (1.0,1.73) node[draw] (4) {};
			\draw (-1.0,1.73) node[draw] (0) {};
			\draw  (0) edge (1);
			\draw  (0) edge (2);
			\draw  (0) edge (3);
			\draw  (0) edge (4);
			\draw  (0) edge (5);
			\draw  (1) edge (2);
			\draw  (1) edge (3);
			\draw  (1) edge (4);
			\draw  (2) edge (3);
			\draw  (2) edge (4);
			\draw  (3) edge (5);
			\draw  (4) edge (5);
		\end{tikzpicture}

\\
\small $\mz$ & \small 3 & \small 3 & \small 3 & \small 2 & \small 2 & \small 2\\
\small $\gamma_\mathbb{Z}$ & \small 4 & \small 4 & \small 4 & \small 3 & \small 3 & \small 3\\
\small $\gamma_\mathbb{R}$ & \small 4 & \small 4 & \small 4 & \small 3 & \small 3 & \small 3\\
\end{tabular}
\end{center}

\begin{center}
\begin{tabular}{ccccccc}
&
		\begin{tikzpicture}[scale=.3,thick]
		\tikzstyle{every node}=[minimum width=0pt, inner sep=2pt, circle]
			\draw (-2.0,0.0) node[draw] (0) {};
			\draw (-1.0,-1.73) node[draw] (1) {};
			\draw (0.99,-1.73) node[draw] (2) {};
			\draw (2.0,0.0) node[draw] (3) {};
			\draw (1.0,1.73) node[draw] (4) {};
			\draw (-1.0,1.73) node[draw] (5) {};
			\draw  (0) edge (1);
			\draw  (0) edge (2);
			\draw  (0) edge (3);
			\draw  (0) edge (4);
			\draw  (0) edge (5);
			\draw  (1) edge (2);
			\draw  (1) edge (3);
			\draw  (1) edge (4);
			\draw  (2) edge (3);
			\draw  (2) edge (4);
			\draw  (3) edge (4);
			\draw  (3) edge (5);
			\draw  (4) edge (5);
		\end{tikzpicture}
&
		\begin{tikzpicture}[scale=.3,thick]
		\tikzstyle{every node}=[minimum width=0pt, inner sep=2pt, circle]
			\draw (-2.0,0.0) node[draw] (1) {};
			\draw (-1.0,-1.73) node[draw] (0) {};
			\draw (0.99,-1.73) node[draw] (2) {};
			\draw (2.0,0.0) node[draw] (3) {};
			\draw (1.0,1.73) node[draw] (4) {};
			\draw (-1.0,1.73) node[draw] (5) {};
			\draw  (0) edge (1);
			\draw  (0) edge (2);
			\draw  (0) edge (3);
			\draw  (1) edge (2);
			\draw  (1) edge (4);
			\draw  (1) edge (5);
			\draw  (2) edge (3);
			\draw  (3) edge (4);
			\draw  (3) edge (5);
			\draw  (4) edge (5);
		\end{tikzpicture}
&
		\begin{tikzpicture}[scale=0.3,thick]
		\tikzstyle{every node}=[minimum width=0pt, inner sep=2pt, circle]
			\draw (-2.0,0.0) node[draw] (0) {};
			\draw (-1.0,-1.73) node[draw] (2) {};
			\draw (0.99,-1.73) node[draw] (1) {};
			\draw (2.0,0.0) node[draw] (3) {};
			\draw (1.0,1.73) node[draw] (4) {};
			\draw (-1.0,1.73) node[draw] (5) {};
			\draw  (0) edge (1);
			\draw  (0) edge (2);
			\draw  (0) edge (4);
			\draw  (0) edge (5);
			\draw  (1) edge (2);
			\draw  (1) edge (3);
			\draw  (1) edge (4);
			\draw  (2) edge (3);
			\draw  (2) edge (5);
			\draw  (3) edge (4);
			\draw  (3) edge (5);
			\draw  (4) edge (5);
		\end{tikzpicture}

\\
\small $\mz$ & \small 2 & \small 2 & \small 2 \\
\small $\gamma_\mathbb{Z}$ & \small 3 & \small 2 & \small 2 \\
\small $\gamma_\mathbb{R}$ & \small 3 & \small 3 & \small 3 \\
\end{tabular}
\end{center}



\end{document}